\documentclass[11pt]{amsart}
\usepackage{}
\usepackage{amsfonts}
\usepackage{amsfonts,latexsym,rawfonts,amsmath,amssymb,amsthm}
\usepackage[plainpages=false]{hyperref}

\usepackage{graphicx}
\makeatletter
\renewcommand\@biblabel[1]{}
\makeatother

\numberwithin{equation}{section}

\newcommand{\beq}{\begin{equation}}
\newcommand{\eeq}{\end{equation}}
\newcommand{\beqs}{\begin{eqnarray*}}
\newcommand{\eeqs}{\end{eqnarray*}}
\newcommand{\beqn}{\begin{eqnarray}}
\newcommand{\eeqn}{\end{eqnarray}}
\newcommand{\beqa}{\begin{array}}
\newcommand{\eeqa}{\end{array}}

\def\lra{\longrightarrow}

\def\bc{\begin{center}}
\def\ec{\end{center}}

\def\begeq{\begin{equation}}
\def\endeq{\end{equation}}
\def\and{\quad{\rm and}\quad}

\let\lra=\longrightarrow

\def\mapright\#1{\,\smash{\mathop{\lra}\limits^{\#1}}\,}

\newtheorem{prop}{Proposition}[section]
\newtheorem{theo}[prop]{Theorem}
\newtheorem{lem}[prop]{Lemma}

\newtheorem{cor}[prop]{Corollary}
\newtheorem{rem}[prop]{Remark}

\newtheorem{defi}[prop]{Definition}

\begin{document}

\date{}
\author   {Yuxing Deng }
\author { Xiaohua $\text{Zhu}^*$}

\thanks {* Partially supported by the NSFC Grants  11271022 and 11331001}

\address{ Yuxing Deng\\School of Mathematical Sciences, Beijing Normal University,
Beijing, 100875, China\\
dengyuxing@mail.bnu.edu.cn}

\address{ Xiaohua Zhu\\School of Mathematical Sciences and BICMR, Peking University,
Beijing, 100871, China\\
 xhzhu@math.pku.edu.cn}

\title{ Steady Ricci solitons  with horizontally $\epsilon$-pinched  Ricci curvature   }
\maketitle

\section*{\ }

\begin{abstract}  In this paper,  we prove that any $\kappa$-noncollapsed gradient steady Ricci soliton  with nonnegative  curvature operator and horizontally $\epsilon$-pinched  Ricci curvature must be rotationally symmetric.   As an application,   we  show  that  any $\kappa$-noncollapsed gradient steady Ricci soliton $(M^n, g,f)$ with   nonnegative curvature operator must be rotationally symmetric if  it admits a unique equilibrium point and its scalar curvature $R(x)$  satisfies $\lim_{r(x)\rightarrow\infty}R(x)f(x)=C_0\sup_{x\in M}R(x)$ with $C_0>\frac{n-2}{2}$.  \end{abstract}

\section {Introduction}

As one of singular model solutions of  Ricci flow,  it is important to   classify   steady Ricci solitons under a suitable curvature condition   \cite{H2}.   In his celebrated paper \cite{Pe},   Perelman conjectured that \textit{all 3-dimensional $\kappa$-noncollapsed steady (gradient) Ricci solitons must be rotationally symmetric}.\footnote{It is proved by Chen  that any  3-dimensional ancient solution has nonnegative sectional curvature \cite{Ch}.}
The conjecture is  solved by Brendle in 2012 \cite{Br1}.   For higher dimensions,  under an extra condition that   the soliton is  asymptotically   cylindrical,    Brendle  also  proves that   any $\kappa$-noncollapsed  steady Ricci soliton  with positive sectional curvature must be rotationally symmetric in \cite{Br2}.  In general, it is still open  \textit{whether an $n$-dimensional $\kappa$-noncollapsed steady Ricci soliton with positive curvature operator is rotationally symmetric for $n\ge 4$}.  For $\kappa$-noncollapsed steady  K\"{a}hler-Ricci  solitons with nonnegative bisectional curvature, the authors have recently proved   that they must be flat \cite{DZ2}, \cite{DZ3}.

 Recall from \cite{Br2},

\begin{defi}\label{brendle} An $n$-dimensional steady  Ricci soliton $(M,g,f)$  is called  asymptotically   cylindrical   if   the following
holds:

(i) Scalar curvature $R(x)$  of $g$ satisfies
$$\frac{C_1}{\rho(x)} \le  R(x) \le  \frac{C_2}{\rho(x)},~\forall~\rho(x)\ge r_0, $$
where $C_1, C_2$  are two positive constants and $\rho(x)$ denotes the distance of $x$  from a fixed  point $x_0$.

(ii) Let $p_m$ be an arbitrary sequence of marked points going to infinity.
Consider  rescaled metrics
$ g_m(t) = r_m^{-1} \phi^*_{r_m t} g,$ where
$r_m R(p_m) = \frac{n-1}{2} + o(1)$ and $ \phi_{ t}$ is a one-parameter subgroup generated by $X=-\nabla f$.   As  $m \to\infty,$
 flows $(M,  g_m(t), p_m)$
converge in the Cheeger-Gromov sense to a family of shrinking cylinders
$(   \mathbb R \times \mathbb S^{n-1}(1), \widetilde g(t)), t \in  (0, 1).$  The metric $\widetilde g(t)$  is given by
\begin{align}\label{n-2-behavior}  \widetilde g(t) =   dr^2+ (n - 2)(2 -2t) g_{\mathbb S^{n-1}(1)},
\end{align}
 where $\mathbb S^{n-1}(1)$ is the unit sphere of euclidean space.
\end{defi}

The purpose of  present  paper is to   give an approach to verify   whether  a $\kappa$-noncollapsed steady Ricci soliton with positive curvature operator is asymptotically   cylindrical.     As  a generalization of $\epsilon$-pinched   Ricci curvature  (cf. \cite{H1}, \cite{CZ}),  we introduce a notion of
horizontally  $\epsilon$-pinched  Ricci curvature as follows.

\begin{defi}
A steady soliton $(M,g,f)$ with positive Ricci curvature is called horizontal   Ricci  curvature $\epsilon$-pinched   if $(M,g,f)$ admits a unique equilibrium point  and
 there exist  $r_0>0$ and $\epsilon>0$ such that  on each level set
 $$\Sigma_r= \{x\in M|~ f(x)=r\},~\forall ~r\ge r_0,$$
   Ricci  curvature  is  $\epsilon$-pinched, i.e.,
\begin{align}\label{pinching along level set}
{\rm \overline{Ric}}(v,v)\geq \epsilon R(x){\overline g}(v,v), ~\forall~ v\in T_{x}\Sigma_{r},
\end{align}
where ${\rm \overline{Ric}}$ is   Ricci curvature of $(\Sigma_r,\overline{g})$  with  the  induced metric $\overline{g}$  on
$\Sigma_r$ as a hypersurface  of $(M,g)$.
\end{defi}

A point $o$  in $(M,g,f)$ is called an  equilibrium one if $\nabla f(o)=0$. Such a point is unique if   Ricci curvature of $(M,g,f)$ is positive.
Moreover, each   $\Sigma_r$ is smooth as long as  $r$ is sufficiently large (cf. Lemma \ref{topology of level set} below).

The following is our main  result in this paper.

\begin{theo}\label{theorem-pinching case}
Any $\kappa$-noncollapsed steady Ricci soliton $(M,g,f)$  with nonnegative curvature operator and  horizontally  $\epsilon$-pinched  Ricci curvature is asymptotically   cylindrical in sense of Brendle.
  As a consequence,  $(M,g,f)$ must be rotationally symmetric.
\end{theo}

Steady Ricci solitons  with $\epsilon$-pinched  Ricci curvature  have been studied by many people \cite{H3}, \cite{CZ}, \cite{Ni}, \cite{DZ1}, etc.. For example,  Ni proves that any  steady Ricci soliton with $\epsilon$-pinched  Ricci curvature and nonnegative sectional curvature must be flat \cite{Ni}. For steady K\"ahler-Ricci solitons, the authors prove that Ni's result is still true even without nonnegative sectional curvature condition \cite{DZ1}.  In order to study asymptotic behavior of  steady Ricci solitons  with  horizontally  $\epsilon$-pinched  Ricci curvature, we will give a  classification of  $\kappa$-solutions  with $\epsilon$-pinched  Ricci curvature in Section 3.

In the proof of Theorem \ref{theorem-pinching case}, we essentially show that  for any sequence $p_{i}\rightarrow\infty$, there exists a subsequence $p_{i_{k}}\rightarrow\infty$ such that
\begin{align*}
(M,g_{p_{i_{k}}}(t),p_{i_{k}})\rightarrow(\mathbb{R}\times\mathbb{S}^{n-1},\widetilde{g}(t),p_{\infty}),~for~t\in(-\infty,1),
\end{align*}
where $g_{p_{i_{k}}}(t)=R(p_{i_{k}})g(R^{-1}(p_{i_{k}})t)$ and $(\mathbb{R}\times\mathbb{S}^{n-1},\widetilde{g}(t))$ is a shrinking cylinders flow, i.e.
$$\widetilde{g}(t)=dr^2+(n-2)[(n-1)-2t]g_{\mathbb{S}^{n-1}}.$$
It is interesting to mention that  as one of steps in  the  proof of Theorem \ref{theorem-pinching case}, we prove that
 scalar curvature  of $g$  decays uniformly (cf. Corollary  \ref{uniform curvature decay}).    We say that    scalar curvature  $R(x)$  of a Riemannian manifold $(M,g)$ decays uniformly if
\begin{align}\label{r-decay-uniform} |R(x)|\rightarrow0, ~as~\rho(x)\rightarrow\infty.
\end{align}
By the way, we also point  that the Ricci curvature pinching condition of $\bar g$ in (\ref{pinching along level set}) in Theorem \ref{theorem-pinching case}  can be replaced by (\ref{m-pinching}) for the ambient metric $g$ if (\ref{r-decay-uniform}) is true (cf. Corollary \ref{theorem-original pinching condition}).

As an application of Theorem \ref{theorem-pinching case}, we  study  $\kappa$-noncollapsed and positively curved steady Ricci solitons with a linear curvature decay,
 \begin{equation}\label{special curvature decay}
\lim_{\rho(x)\rightarrow\infty}R(x)f(x)=C_0\sup_{x\in M}R(x),
\end{equation}
 where  $C_0>0$ is a constant.   It is known  that there are constants $c_1, c_2>0$ such that
 \begin{align}\label{cao}c_1 \rho(x)\le f(x)\le c_2 \rho(x),~\rho(x)\ge r_0,
 \end{align}
 if a  steady Ricci soliton has positive Ricci curvature and an equilibrium point (cf. \cite{CaCh}).
By  check of      pinching condition of  horizontal  Ricci curvature,   we are able to prove

\begin{theo}\label{main-theorem-1} Any
$n$-dimensional $\kappa$-noncollapsed steady Ricci soliton $(M,$
\newline $g,f)$ with nonnegative  curvature operator and a unique equilibrium point must be rotationally symmetric if   scalar curvature $R(x)$ of $g$  satisfies (\ref{special curvature decay}) with $C_0>\frac{n-2}{2}$.
\end{theo}

The paper is organized as follows.  In Section 2, we give a classification of  $\kappa$-solutions with $\epsilon$-pinched  Ricci curvature. Theorem \ref{theorem-pinching case} and
 Theorem \ref{main-theorem-1} will be proved in Section 3, 4 respectively.  Some related estimates  for steady  Ricci solitons  are used  from  \cite{DZ2}.

\noindent {\bf Acknowledgements.}   The authors   are  grateful  to  referees  for   a careful reading and  many valuable suggestions to  the  paper.   Without their  help, the present version of paper couldn't  be finished.

\section{Classification of Ricci pinched $\kappa$-solutions}

Throughout the paper, we say that $(M,g,f)$ is a (gradient) steady Ricci soltion if
$$R_{ij}=\nabla_{i}\nabla_{j}f.$$
 The following lemma shows that   each level set $\Sigma_r$  of $(M,g,f)$  is diffeomorphic to a sphere if   it  admits an equilibrium point
 and has positive Ricci curvature.

\begin{lem}\label{topology of level set}
Let $(M,g,f)$ be an $n$-dimensional noncompact steady  Ricci soliton with positive Ricci curvature and an equilibrium point $o$. Then  any  $\Sigma_r$ is diffeomorphic to $\mathbb{S}^{n-1}$, where $r>f(o)$.
\end{lem}

\begin{proof}
By  soliton equation and  positivity of Ricci curvature,  $f$ is strictly convex.  Moreover,  the equilibrium point is unique. Thus by  the Morse Lemma,  there exists   a coordinates system   $(x_{1},\cdots,x_{n})$ in a neighbourhood $U$ near $o$ such that
$$f(x)=f(o)+x_{1}^{2}+\cdots+x_{n}^{2}.$$
 It  follows that $\Sigma_{f(o)+\epsilon}=\{x\in M| ~f(x)=f(o)+\epsilon\}$ is diffeomorphic to $\mathbb{S}^{n-1}$ for  any small $\epsilon>0$.
  Since $\Sigma_r$ is evolved along the gradient flow of $-\nabla f$,  each  $\Sigma_r$ is diffeomorphic to  $\Sigma_{f(o)+\epsilon}$, where $r>f(o)+\epsilon$. The lemma is proved.
\end{proof}

This lemma will play an important role in the analysis of the asymptotic behavior of steady Ricci solitons below. In this section, we give a classification of steady  Ricci solitons with $\epsilon$-pinched  Ricci curvature. Recall

\begin{defi}
A Riemannian manifold $(M,g)$ is called  Ricci  curvature $\epsilon$-pinched,  if it has nonnegative Ricci curvature and its Ricci curvature satisfies
\begin{align}
{\rm Ric}(x)\geq \epsilon R(x)g(x),
\end{align}
where $\epsilon$ is a positive constant independent of $x\in M$.  Similarly, a Ricci flow $(M,g(t))$ on $t\in (a,b)$ is called  Ricci  curvature $\epsilon$-pinched,   if it has nonnegative Ricci curvature along the flow and its  Ricci  curvature satisfies
\begin{align}
{\rm Ric}(x,t)\geq \epsilon R(x,t)g(x,t),~\forall  ~t\in (a,b).
\end{align}
where $\epsilon$ is a positive constant independent of $x\in M$ and $t\in (a,b)$.
\end{defi}

we prove

\begin{theo}\label{classification of flow}
Let $(M,g(t))$ be a simply connected and Ricci curvature  $\epsilon$-pinched $\kappa$-solution.  Then $(M,g(t))=(M_{1},g_{1}(t))\times(M_{2},g_{2}(t))\times\cdots\times (M_{k},g_{k}(t))$, and each $(M_{i},g_{i}(t))$ is an  Einstein metrics  flow on a simply connected and compact symmetric space.
 More precisely,  each $(M_{i},g_{i}(t))$ is   one of the following three types:
\begin{enumerate}
\item[(i)] $(\mathbb{S}^{n_{i}}, g(t))$  is  a  Ricci flow with positive  constant sectional curvature;
\item[(ii)] $(\mathbb{CP}^{n_{i}} , g(t))$ is a  Ricci flow with  constant  positive  bisectional curvature;
\item[(iii)]$(M_i, g(t))$  is   an Einstein metrics flow  on an  irreducible symmetric space except (i) and (ii).
\end{enumerate}
\end{theo}

To prove Theorem \ref{classification of flow}, we need the following classification result.

\begin{lem}\label{lem-einstein and shrinking}
\begin{enumerate}\ \\
\item[i)] A simply connected Einstein manifold with positive curvature operator must be $\mathbb{S}^{n}$ with constant sectional curvature.
\item[ii)] A simply connected shrinking soliton with positive curvature operator must be $\mathbb{S}^{n}$ with constant sectional curvature.
\item[iii)] A  simply connected K\"{a}hler-Einstein manifold with positive bisectional curvature must be $\mathbb{CP}^{n}$ with constant bisectional curvature.

\item[iv)] A simply connected K\"{a}hler-Ricci soliton with positive bisectional curvature must be $\mathbb{CP}^{n}$ with constant bisectional curvature.
\end{enumerate}
\end{lem}

\begin{proof}
i) It is clear that the manifold   is compact. Let  $g(t)=(1-2\lambda t)g$, where $n\lambda$ is  scalar curvature of $g$. Then  $g(t)$ satisfies Ricci flow equation,
 \begin{align}\label{eq:1}
         \frac{\partial g(t)}{\partial t} &= -2{\rm Ric}(g(t)),
                  \end{align}
By Wilking's result \cite{Wi},   there exists $r(t)$ such that
\begin{align}
r(t)g(t)\rightarrow g_{\mathbb{S}^{n}}, ~{\rm as}~t\to\frac{1}{2\lambda},
\end{align}
where $g_{\mathbb{S}^{n}}$ is the metric with constant sectional curvature $1$ on  $\mathbb{S}^{n}$.
 Thus  $g$ must be  a metric with constant sectional curvature on $\mathbb S^n$.

ii) By Munteanu and Wang's result \cite {Wa},  any shrinking soliton with nonnegative  curvature and positive Ricci curvature must be compact. Then by the  same argument as in i), the manifold should be isometric to one with constant sectional curvature on $\mathbb S^n$.

iii) and iv) are known, since the manifold  is compact and so it is  biholomorphic to $\mathbb{CP}^{n}$ by Frankel's conjecture. By the uniqueness of   K\"{a}hler-Einstein metrics \cite{BM}, $(M,g)$  has constant bisectional curvature. Another proof can come from  \cite{CST} by using the Ricci flow.

\end{proof}

\begin{proof}[Proof of  Theorem \ref{classification of flow}]
By Theorem 11.3 for $\kappa$-solutions  in \cite{Pe}, there exists a sequence $(M,
\tau_{i}^{-1}g(t-\tau_i),p_{\tau_i})$  ($\tau_{i}\to\infty$ )  converging  to a limit $(M_{\infty},g_{\infty}(t),p_{\infty})$, which is a non-flat shrinking solitons  solution with nonnegative curvature operator.   Since the pinching condition is preserved under rescaling, $(M_{\infty},g_{\infty}(t))$ is also $\epsilon$-Ricci pinched.   We claim that   Ricci curvature of $g_{\infty}(t)$ is strictly positive.
If the claim is not true,  Ricci curvature of $(M_{\infty},g_{\infty}(t))$ vanishes along some direction somewhere.  Then scalar curvature of the soliton is zero at some point $p\in M_\infty$ and some time  $t_0$ by    pinched $\epsilon$-Ricci  curvature property.   We may assume that $(M_{\infty},g_{\infty}^{\prime}(\tau))$ is  a  Ricci flow generated by the non-flat shrinking Ricci soliton $(M_{\infty},g_\infty(t_0))$ which satisfies
 \begin{align}
         \frac{\partial g_{\infty}^{\prime}(\tau)}{\partial \tau} &= -2{\rm Ric^{\prime(\infty)}}(\tau),~\tau\in (-\infty,1),\notag
                  \end{align}
and  scalar curvature $R^{\prime(\infty)}(\tau)$ is zero at $(p,t_0')$.
It turns  that  $R^{\prime(\infty)}(\tau)$ satisfies  evolution equation
$$\frac{\partial R^{\prime(\infty)}(\tau)}{\partial \tau}=\Delta R^{\prime(\infty)}(\tau) +2|{\rm Ric^{\prime(\infty)}}(\tau)|^{2}.$$
Note that $R^{\prime(\infty)}(\tau)$ attains its minimum  at $(p,t_0')$ for some $t_0'<1$.  Thus
 $R^{\prime(\infty)}(t)\equiv 0$ by the maximum principle. Namely,  $(M_{\infty},g_{\infty}(t_0))$ is flat. This is impossible!

 Applying Munteanu and Wang's result  for  shrinking solitons  with nonnegative  sectional curvature and positive Ricci curvature  \cite{Wa}, we know that $(M_{\infty},g_{\infty}(t))$ is compact.
Next we show that $M$ is also compact and it is diffeomorphic to $M_{\infty}$. Let $g_{\tau}=\tau^{-1}g(t_{0}-\tau)$, for $\tau\in (1,\infty)$. By the convergence, there exists a sequence of isometries $\Phi_{\tau_i}:B(p_{\tau_i},r_{\tau_i}; g_{\tau_i})\rightarrow \Phi_{\tau_i}(B(p_{\tau_i},r_{\tau_i}); g_{\tau_i})(\subset M_{\infty})$ such that $\Phi_{\tau_i}(B(p_{\tau_i},r_{\tau_i}; g_{\tau_i}))$ exhaust $M_{\infty}$ as $r_{\tau_i}\rightarrow\infty$. Since $M_{\infty}$ is compact, $\Phi_{\tau_i}(B(p_{\tau_i}, r_{\tau_i}; g_{\tau_i}))= M_{\infty}$, for $\tau$ large enough. Hence, $B(p_{\tau_i},r_{\tau_i}; g_{\tau_i})$ is diffeomorphic to $M_{\infty}$ and it is compact without boundary. As a consequence, $B(p_{\tau_i},r_{\tau_i}; g_{\tau_i})$ is a complete manifold and it must cover $M$. Hence  $M$ is compact and  diffeomorphic to $M_{\infty}$.
By a classification theorem of  compact manifolds with nonnegative curvature operator \cite{Wi},  $(M,g(t))$ is isometric to flow $(M_{1},g_{1}(t))\times(M_{2},g_{2}(t))\times\cdots\times (M_{k},g_{k}(t))$, where each $(M_{i},g_{i}(t))$ is one of the following three types:
\begin{enumerate}
\item[(i)]  Ricci flow on $\mathbb{S}^{n_{i}}$ with positive curvature operator;
\item[(ii)]Ricci flow on  $\mathbb{CP}^{n_{i}}$ with positive sectional curvature;
\item[(iii)]  Einstein metrics flow on a  symmetric space   except (i) and (ii).
\end{enumerate}
It remains  to deal with case (i) and (ii).
For case (i), we may assume that $M=\mathbb{S}^{n}$.   Then $(M_{\infty},g_{\infty}(t))$ is a round spheres flow by Lemma \ref{lem-einstein and shrinking}.  Thus   $(M,g_{\tau_i},p_{\tau_i})$  converge to a  round sphere as $\tau_i\to \infty$. On the other hand,   by Theorem 3.1 in \cite{Hu} (also see \cite{BW}),  the curvature pinching property  is preserved along the flow $(M,g(t))$.  Hence  $(M,g(t))$ is getting more and more round from largely  negative $t$.  Therefore, $(M,g(t))$ must be  a round sphere for all $t$.

For case (ii), we may assume that $M=\mathbb{CP}^{n}$.  Then $g(t)$ are all K\"ahler metrics.    Moreover, there is some $C_0>0$   such that  K\"ahler classes of  $\hat g(t)=C_0^{-1} e^t g( C_0(1-e^{-t}))$  are all $2\pi c_1(\mathbb{CP}^{n})$. It follows that
 $\hat g(t)$ satisfies the normalized K\"ahler-Ricci flow,
$$\frac{\partial \hat g}{\partial t}= -{\rm Ric}(\hat g)+\hat g,~t\in ~(-\infty,\infty).$$
By the convergence of $g_\tau$ and Lemma  \ref{lem-einstein and shrinking},  it is easy to see  that  there exists a sequence of  $\hat g(t_i)$ which converges to the Fubini-Study metric of  $\mathbb{CP}^{n}$ as $t_i\to -\infty$. Now we can apply the stability
result  for K\"ahler-Ricci flow near a  K\"{a}hler-Einstein metric in \cite{Zh} to conclude that $\hat g(t)$ is the the Fubini-Study metric for any $t$.  Hence, $ g(t)$ are all K\"ahler metrics with   positive  constant bisectional curvature.

\end{proof}

\section{Proof of Theorem \ref{theorem-pinching case}}

In  this section,  we  prove   Theorem \ref{theorem-pinching case}  under the
 condition (\ref{pinching along level set}). First,  we  recall  a result  of  asymptotic  behavior  for  $\kappa$-noncollapsed steady Ricci solitons with nonnegative curvature operator proved in Theorem 1.6, \cite{DZ2}.

\begin{theo}\label{theorem-soliton-real}
Let $(M,g, f)$ be an $n$-dimensional noncompact $\kappa$-noncollapsed steady Ricci soliton with a unique equilibrium point. Suppose that $M$ has nonnegative curvature operator and positive Ricci curvature.
Then, for any $p_{i}\rightarrow\infty$, the sequence of rescaled flows $(M,R(p_{i})g(R^{-1}(p_{i})t),p_{i})$ converges subsequently to a Ricci flow
$(\mathbb{R}\times N,\widetilde{g}(t))$ ( $t\in (-\infty,0]$) in the Cheeger-Gromov topology, where
$$\widetilde{g}(t)={\rm d}s\otimes{\rm d}s+g_{N}(t),$$
and $(N,g_{N}(t))$ is a   $\kappa$-noncollapsed Ricci flow  with nonnegative curvature operator  on $N$ with dimension $n-1$.
\end{theo}

\begin{rem}\label{soliton-lower-bound-estimate} The proof of  Theorem \ref{theorem-soliton-real} is based on an estimate
 \begin{align}\label{maximal-growth}
\frac{C}{\rho(x)}\leq R(x),~{\rm if}~\rho(x)\ge r_0>>1,
\end{align}
where $C>0$ is a uniform constant.
(\ref{maximal-growth}) is proved for   $\kappa$-noncollapsed  K\"ahler-Ricci solitons with nonnegative bisectional curvature in  Proposition 4.3 in \cite{DZ2}. For steady Ricci solitons  in   Theorem \ref{theorem-soliton-real} it  is still true
since we assume the existence of   equilibrium points.\footnote{It is proved that  there exists an equilibrium point on a steady  K\"ahler-Ricci soliton with nonnegative bisectional curvature and positive Ricci curvature  \cite{DZ1}.}   (\ref{maximal-growth}) will be frequently used   below (cf. (\ref{r-lower})).

\end{rem}

Under horizontally  $\epsilon$-pinched Ricci curvature,  we can  also control the local structure of steady Ricci solitons.

\begin{lem}\label{lem-for the theorem of pinching case-1}Let  $(M,g,f)$ be a $\kappa$-noncollapsed steady Ricci soliton  as in Theorem \ref{theorem-pinching case}.  Let  $p_{r}\in \Sigma_{r}$ such that
$R(p_{r})=\inf_{x\in \Sigma_{r}}R(x)$. Then, for any $k\in\mathbb{N}$, there exists $r(k,\epsilon)$ such that for any $r\geq r(k,\epsilon)$
\begin{align}\label{set-mr-contain}
B(p_{r},\frac{k}{\sqrt{R_{max}}};  g_{r})\subset M_{r,k}^{\prime}\subset B(p_{r},2\pi\sqrt{\frac{n-2}{\epsilon}}+\frac{2k}{\sqrt{C_{0}}};  {g_{r}}),
\end{align}
where $g_{r}=R(p_{r})g$, $R_{\max}=\max_{x\in M}R(x)$, $C_{0}=R_{\max}-\sup_{x\in \Sigma_{f(o)+1}}R(x)>0,$ and $M_{r,k}^{\prime}$  is a subset of $M$ defined by
$$M_{r,k}^{\prime}=\{x\in M| ~ r-\frac{k}{\sqrt{R(p_{r})}}\le f(x)\le r+\frac{k}{\sqrt{R(p_{r})}},R(p_{r})=\inf_{x\in\Sigma_r}R(x)\}.$$
\end{lem}

\begin{proof}

Let $\overline{g}= g|_{\Sigma_{r}}$ be  an  induced metric of   hypersurface $\Sigma_r$  of $(M,g)$ and
$\overline{{\rm Ric}}(\bar g)$    Ricci  curvature of $\bar g$ with components $\overline R_{i'j'}$, where indices $i',j'$ are corresponding to  a basis of vector fields on $\Sigma_r$.
By (\ref{pinching along level set}), we have
\begin{align*}
\overline{R}_{i'j'}\geq \epsilon R\overline{g}_{i'j'}\geq \epsilon R(p_{r})\overline{g}_{i'j'}, ~\forall ~x\in \Sigma_{r},
\end{align*}
as long as  $r$ is  large enough.   By the Myer's theorem,  the diameter of $\Sigma_{r}$ is bounded by
\begin{align}\label{diameter-mr}
{\rm diam}(\Sigma_{r},g)\le {\rm diam}(\Sigma_{r},\overline{g}_{r})\leq 2\pi\sqrt{\frac{n-2}{\epsilon R(p_{r})}}.
\end{align}
It follows
\begin{align}\label{mr-set}
\Sigma_{r}\subset B(p_{r},A(\epsilon);  {g_{r}}),
\end{align}
where $A(\epsilon)=2\pi\sqrt{\frac{n-2}{\epsilon}}$.

Let $\phi_{t}$ be a one parameter subgroup generated by $-\nabla f$. Then  there exists a largely  negative  $t_0$ such that (cf. Lemma 4.1, \cite{DZ2}),
\begin{align}\label{length-growth} c_1 |t|\le r(\phi_t(p), o)\le c_2|t|,~\forall ~t\le t_0,
\end{align}
where the uniform constant $c_1$ and $c_2>0$ independent of $p\in \Sigma_{f(o)+1}.$
Moreover, by the identity,
\begin{align}\label{identity-r}
|\nabla f|^{2}+R=R_{\max},
\end{align}
where $R_{\max}=R(o), $ we have
\begin{align}\label{monoto-scalar}
\frac{d R(\phi_{t}(q))}{d t}=2{\rm Ric}(\nabla f,\nabla f)>0, ~\forall ~q\in M\setminus\{o\}.
\end{align}
(\ref{monoto-scalar})  means that scalar curvature is decreasing along the  integral curves $\phi_{t}(q)$.
Thus  by (\ref{length-growth}),  there exists   a large $r_0$ such that for any  $x$ with $\rho(x,o)\ge r_0$
\begin{align}\label{curvature-bounded}0\le R(x)\le \sup_{z\in \Sigma_{f(o)+1}}R(z)<R_{\max}.
\end{align}
On the other hand, by Remark \ref{soliton-lower-bound-estimate} together with (\ref{cao}),  there exists $C>0$ such that
\begin{align}\label{r-lower}
R(p_{r})\cdot r\geq C>0.
\end{align}
 Then  for a fixed $k$, we have
\begin{align}
\frac{r-\frac{k}{\sqrt{R(p_{r})}}}{r}\rightarrow1,~as~r\rightarrow\infty.\notag
\end{align}
Since
$$C_{1} \rho(x,o)\leq f(x)\leq C_{2} \rho(x,o),$$
we see
 $$\rho(x,o)\to \infty, ~{\rm as }~r\to \infty, ~\forall ~x\in  M_{r,k}^{\prime}.$$
Hence by  (\ref{curvature-bounded}), we get
$$
0\le R(x)\le \sup_{z\in \Sigma_{f(o)+1}}R(z)<R_{\max},~x\in M_{r,k}^{\prime}$$
as long as $r$ is large enough. This implies
\begin{align}\label{nabla-f}
\sqrt{C_{0}}\le|\nabla f|(x)\le \sqrt{R_{\max}},~~x\in M_{r,k}^{\prime},
\end{align}
where $C_{0}=R_{\max}-\sup_{x\in \Sigma_{f(o)+1}}R(x)>0$.

 For any $q\in M_{r,k}^{\prime}$, there exists  $q^{\prime}\in \Sigma_{r}$  such that $\phi_{s}(q)=q^{\prime}$ for some $s\in \mathbb{R}$.
Then by (\ref{mr-set}) and  (\ref{nabla-f}), we have
\begin{align*}
d(q,p_{r})\leq& d(q^{\prime},p_{r})  + d(q,q^{\prime})\\
\leq& {\rm diam}(\Sigma_{r},g)+\mathcal{L}(\phi_{\tau}|_{[0,s]})\\
\leq& 2\pi\sqrt{\frac{n-2}{\epsilon R(p_{r})}}+|\int_{0}^{s}|\frac{d\phi_{\tau}(q)}{d\tau}|d\tau|\\
=& 2\pi\sqrt{\frac{n-2}{\epsilon R(p_{r})}}+\int_{0}^{s}|\nabla f(\phi_{\tau}(q))|d\tau\\
\le& 2\pi\sqrt{\frac{n-2}{\epsilon R(p_{r})}}+\int_{0}^{s}|\nabla f(\phi_{\tau}(q))|^{2}\cdot \frac{1}{\sqrt{C_{0}}}d\tau\\
=& 2\pi\sqrt{\frac{n-2}{\epsilon R(p_{r})}}+|\int_{0}^{s}\frac{d(f(\phi_{\tau}(q)))}{d\tau}\cdot \frac{1}{\sqrt{C_{0}}}d\tau|\\
\leq& 2\pi\sqrt{\frac{n-2}{\epsilon R(p_{r})}}+|f(q)-f(p_{r})|\cdot \frac{1}{\sqrt{C_{0}}}\\
\leq& \Big(2\pi\sqrt{\frac{n-2}{\epsilon}}+\frac{2k}{\sqrt{C_{0}}}\Big)\cdot \frac{1}{\sqrt{R(p_{r})}}.
\end{align*}
Thus
\begin{align*}
M_{r,k}^{\prime}\subset B(p_{r},A(\epsilon)+\frac{2k}{\sqrt{C_{0}}} ; g_{r}).
\end{align*}
 This proves the second relation in (\ref{set-mr-contain}).

For any $q\in M$, let $\gamma(s)$ be any curve connecting $p_{r}$ and $q$ such that $\gamma(s_{1})=q$ and $\gamma(s_{2})=p_{r}$. Then,
\begin{align*}
\mathcal{L}(q,p_{r})=&\int_{s_{1}}^{s_{2}}\sqrt{\langle \gamma^{\prime}(s),\gamma^{\prime}(s)\rangle}ds\\
\geq& \int_{s_{1}}^{s_{2}}\frac{|\langle\gamma^{\prime}(s),\nabla f\rangle|}{|\nabla f|}ds\\
\geq& \frac{1}{\sqrt{R_{max}}}|\int_{s_{1}}^{s_{2}}\langle\gamma^{\prime}(s),\nabla f\rangle ds|\\
=&\frac{1}{\sqrt{R_{max}}}|f(p_{r})-f(q)|.
\end{align*}
It follows
\begin{align}
d(q,p_{r})\geq \frac{1}{\sqrt{R_{max}}}|f(p_{r})-f(q)|.\notag
\end{align}
In particular,  for  $q\in M\setminus M_{r,k}^{\prime}$, we get
\begin{align}
d(q,p_{r})\geq \frac{1}{\sqrt{R_{max}}}\cdot\frac{k}{\sqrt{R(p_{r})}}.\notag
\end{align}
Hence
\begin{align}\label{relation-1}
B_{g_{r}}(p_{r},\frac{k}{\sqrt{R_{max}}})\subset M_{r,k}^{\prime}.
\end{align}
The first relation  in (\ref{set-mr-contain}) is also true.
\end{proof}

\begin{lem}\label{pinching of scalar curvature}
Under the condition of Theorem \ref{theorem-pinching case}, there exists a constant $C(\epsilon)>0$ independent of $x$, $y$ and $r$ such that
\begin{align}
\sup_{x,y\in \Sigma_r}\frac{R(x)}{R(y)}\le C(\epsilon),~for~r>f(o).
\end{align}
\end{lem}

\begin{proof}
Let $p_{r}$ be chosen as  same as in Lemma \ref{lem-for the theorem of pinching case-1} for $r>f(o)$. We consider  rescaled  $\kappa$-solutions  $(M,R(p_{r})g(R(p_{r})^{-1}t),p_{r})$.  As in the proof of Theorem 3.3 in \cite{DZ2}, we see that for any $d>0$  there exists a constant $C(d)>0$ such that
\begin{align}\label{bounded-curvature}
R_{R(p_{r})g}(x)\le C(d), ~\forall~x\in ~B(p_{r},d;  R(p_{r})g).
\end{align}
On the other hand,  by (\ref{mr-set}),  there exists  a constant $r(\epsilon)$ such that for any $r\ge r(\epsilon)$ it holds
\begin{align}\label{sigma-set} \Sigma_{r}\subset B(p_{r},2\pi\sqrt{\frac{n-2}{\epsilon}}  ;  R(p_{r})g).
\end{align}
Then  there exists a  $C(\epsilon)$ such that
$$\frac{R(x)}{R(y)}\le \frac{R(x)}{R(p_{r})}\leq C(\epsilon),~\forall ~x,y\in \Sigma_{r}, ~\forall~r\ge r(\epsilon).$$
Since $\{x\in M|~  f(x)\le r(\epsilon)\}$ is compact, the lemma is true.
\end{proof}

\begin{lem}\label{lem-curvature decay}
Under the condition of Theorem \ref{theorem-pinching case}, for any $x\in M\setminus\{o\}$,
 $$R(\phi_{t}(x))\rightarrow0, ~as~t\rightarrow-\infty.$$
\end{lem}

\begin{proof}
If the lemma is not true, by (\ref{monoto-scalar}),  there exists  $x_0$ such that
\begin{align}
\lim_{t\rightarrow-\infty}R(x_{0},t) =\lim_{t\rightarrow-\infty}R(\phi_{t}(x_{0}))=C_0>0.\notag
\end{align}
Let $r_t=f(\phi_{t}(x_{0}))$ and $p_{r_t}$ be chosen  as in Lemma \ref{lem-for the theorem of pinching case-1}. By Lemma \ref{pinching of scalar curvature}, $R(p_{r_t})\ge C_0/C(\epsilon)$.
On the other hand, for any $x\in \{y\in M| ~ f(y)\ge f(\phi_{-1}(x_{0}))\}$, there exists $t_x\le -1$ such that $f(x)=f(\phi_{t_x}(x_{0}))$. Thus
\begin{align}
R(x)\ge R(p_{r_{t_x}})\ge\frac{C_0}{C(\epsilon)}.\notag
\end{align}
This implies that there exists a uniform constant $C_0'$ such that
\begin{align}\label{positive-bound}
R(x)\ge C_0',~\forall~x\in M.
\end{align}
 However, by Lemma 4.3 in \cite{D}, we know
\begin{align}
\frac{1}{{\rm Vol}(B(o,r))}\int_{B(o,r)}R(x)dv\le \frac{C}{r}, ~\forall ~r>0,\notag
\end{align}
for some uniform  $C$  independent of $r$. This is   a contradiction of  (\ref{positive-bound}).  The lemma is proved.
\end{proof}

As a corollary of  Lemma \ref{pinching of scalar curvature} and Lemma \ref{lem-curvature decay}, we prove

\begin{cor}\label{uniform curvature decay} Let   $(M,g,f)$ be a  $\kappa$-noncollapsed steady Ricci soliton  with nonnegative curvature operator.   If   $g$  has  horizontally   $\epsilon$-pinched  Ricci curvature,
 scalar curvature  $R(x)$ of $g$  decays uniformly. Namely,  $R(x)$ satisfies (\ref{r-decay-uniform}).
\end{cor}

\begin{proof}
 We suffice to prove that $R(y_i)\to 0$ for  any sequence of $\{y_i\}$ with $f(y_i)\rightarrow\infty$.  Let $r_i=f(y_i)$ and $p_{r_i}$ be defined
  as in Lemma \ref{lem-for the theorem of pinching case-1}. By Lemma \ref{pinching of scalar curvature},
\begin{align}
\frac{R(y_i)}{R(p_{r_{i}})}\leq C(\epsilon).
\end{align}
On the other hand, for a fixed $x_0$, there is a $t_i$ such that $f(\phi_{t_i}(x_{0}))=r_i= f(p_{r_i})$. Thus by
 Lemma \ref{lem-curvature decay}, we have
\begin{align}
R(p_{r_{i}})\le R(\phi_{t_i}(x_{0}))\to 0, ~as~r_i\rightarrow\infty.\notag
\end{align}
Hence
$R(y_{i})\rightarrow 0$ as $i\rightarrow\infty.$

\end{proof}

\begin{rem}\label{remark} By Corollary \ref{uniform curvature decay},  the constant $C_0$ at  relation (\ref{set-mr-contain}) in Lemma \ref{lem-for the theorem of pinching case-1} can be chosen  by $R_{max}$ (cf. Lemma \ref {nabla-f} below).
\end{rem}

Combining Theorem \ref{theorem-soliton-real}, Lemma \ref{lem-for the theorem of pinching case-1}
and Corollary \ref{uniform curvature decay}, we prove

\begin{lem}\label{lem-for the theorem of pinching case-2}Let  $(M,g,f)$ be a $\kappa$-noncollapsed steady Ricci soliton  as in Theorem \ref{theorem-pinching case} and   $p_{r}\in \Sigma_{r}$  chosen as in Lemma  \ref{lem-for the theorem of pinching case-1}.   Then
for any sequence  of $r\rightarrow\infty$, there exists a subsequence $r_{i}\rightarrow\infty$ such that
\begin{align*}
(M,g_{r_{i}}(t),p_{r_{i}})\rightarrow(\mathbb{R}\times\mathbb{S}^{n-1},\widetilde{g}(t),p_{\infty}),~for~t\in(-\infty,0],
\end{align*}
where $g_{r_{i}}(t)=R(p_{r_{i}})g(R^{-1}(p_{r_{i}})t)$ and $(\mathbb{R}\times\mathbb{S}^{n-1},\widetilde{g}(t))$ is a shrinking cylinders flow, namely,
$$\widetilde{g}(t)=ds\otimes ds+(n-2)[(n-1)-2t]g_{\mathbb{S}^{n-1}}.$$
\end{lem}

\begin{proof}
By Theorem \ref{theorem-soliton-real}, for any sequence  of $r\rightarrow\infty$, there exists a subsequence of  $r_{i}\rightarrow\infty$ such that
\begin{align*}
(M,g_{r_{i}}(t),p_{r_{i}})\rightarrow(M_{\infty},\widetilde g(t),p_{\infty}),~for~t\in(-\infty,0],
\end{align*}
where  $(M_{\infty},  \widetilde g(t))=(\mathbb{R} \times N,  ,ds^2+g_{N}(t))$ and $g_{N}(t)$ satisfies the Ricci flow equation for $t\in (-\infty,0]$. In the following,
 we  first  need to show that $N$ is diffeomorphic to   $\mathbb{S}^{n-1}$.  Fix a point $x_0\in M\setminus\{o\}$ and let $x_\tau=\phi_\tau(x_0)$  such that $f(x_{\tau_i})=r_i$.  We are going to construct diffeomorphism $\Phi_{r_{i}}:M_{r_i,k}^{\prime}\rightarrow \Phi_{r_{i}}(M_{r_i,k}^{\prime})( \subset  \mathbb{R} \times \mathbb{S}^{n-1} )$ and show that $(\Phi_{r_{i}}(M_{r_i,k}^{\prime}),(\Phi^{-1}_{r_{i}})^{\ast}g_{r_i}(0), \Phi_{r_{i}}(p_{r_i}))$ subsequently converge to $(  \mathbb{R} \times \mathbb{S}^{n-1}, g_{\infty}(0),p_{\infty})$ in $C_{loc}^{\infty}$ sense, where $M_{r,k}^{\prime}$  is a subset of $M$ defined as in Lemma \ref{lem-for the theorem of pinching case-1}.

By Lemma \ref{topology of level set}, we know that  the level set $\Sigma_{f(o)+1}$ is diffeomorphic to  $\mathbb{S}^{n-1}$. For any $x\in M$,  there exists  a unique $t_x$ and $\overline{x}\in \Sigma_{f(o)+1}$ such that $x=\phi_{t_x}(\overline{x})$.  We define $\Phi_{r_{i}}(x)=((f(x)-f(x_{\tau_i}))\sqrt{R(p_{r_i})},   \overline{x})\in \mathbb{R}\times \Sigma_{f(o)+1} \doteq   \mathbb{R} \times \mathbb{S}^{n-1}$.
  Since $|\nabla f|>0$ on $M\setminus\{o\}$, $\Phi_{r_{i}}$ is essentially a diffeomorphism from   $M_{r_i,k}^{\prime}$   to $\mathbb{R}\times\mathbb{S}^{n-1}\mathbb{R}$. Note that $\Phi_{r_{i}}( M_{r_i,k}')$ exhausts  $  \mathbb{R} \times \mathbb{S}^{n-1}$ as $k\rightarrow\infty$ by (\ref{diameter-mr}).
 On the other hand,   by  Lemma \ref{lem-for the theorem of pinching case-1},   there exist constants $C_1$ and $C_2$ depends only on $\epsilon$ and  $n$ such that
\begin{align}\label{control of set addition}
B(p_{r_i},C_1^{-1} k;  g_{r_i}(0))\subset  M_{r_i,k}^{\prime}\subset B(p_{r_i},C_2 k;  g_{r_i}(0)).
\end{align}
 Thus $\Phi_{\tau_{i}}(B(p_{r_i}, k;  g_{\tau_i}(0)))$ also exhausts $  \mathbb{R} \times \mathbb{S}^{n-1}$.
As a consequence,   \newline  $(\Phi_{r_{i}}(M_{r_i,k}^{\prime}),(\Phi^{-1}_{r_{i}})^{\ast}g_{r_i}(0),    \Phi_{r_{i}}(p_{r_i}))$ subsequently converges to $(  \mathbb{R} \times \mathbb{S}^{n-1}, \widetilde g(0),$ \newline $p_{\infty})$ in $C_{loc}^{\infty}$ sense.
Hence $  \mathbb{R} \times \mathbb{S}^{n-1}$ is diffeomorphic to $M_\infty$.
Therefore, to prove  that $N$ is diffeomorphic to $\mathbb{S}^{n-1}$,  it reduces  to show that the limit flow  $\widetilde g(t)$ splits along the direction  of $\mathbb{R}$ in $\mathbb{R}\times\mathbb{S}^{n-1}$.

Fix any $d_{0}>0$.  Let $g_{r_{i}}=g_{r_{i}}(0)$ and  $X_{(i)}=R(p_{r_i})^{-\frac{1}{2}}\nabla f$.
 Then as in (\ref{bounded-curvature}),  it holds
$$| {\rm Ric}_{g_{r_i}}|_{g_{r_i}}(x)\le C(d_0),~\forall~x\in B(p_{r_i},d_{0} ;  {g_{r_i}}).$$
Namely,
$$|{\rm  Ric}|(x)\le C(d_0)R(p_{r_i}),~\forall~x\in B(p_{r_i},d_{0} ;  {g_{r_i}}).$$
By  Corollary \ref{uniform curvature decay},  it follows
\begin{align}\label{x-parallel}
\sup_{ B(p_{r_i},d_{0} ;  {g_{r_i}})}| \nabla_{(g_{r_i})}X_{(i)}|_{g_{r_i}}&= \sup_{ B(p_{r_i},d_{0} ;  {g_{r_i}})}\frac{|{\rm Ric}|}{\sqrt{R(p_{r_i})}}\notag\\
&\le C(d_0)\sqrt{R(p_{r_i})} \to 0.
\end{align}
On the other hand,  by Shi's higher order estimate, we have
$$\sup_{ B(p_{r_i},d_{0} ;  {g_{r_i}})}| \nabla^{m}_{(g_{r_i})}X_{(i)}|_{g_{r_i}}\leq C(n)\sup_{ B(p_{r_i},d_{0} ;  {g_{r_i}})}| \nabla^{m-1}_{(g_{r_i})}{\rm Ric}({g_{r_i})}|_{g_{r_i}}\le C_3.$$
Thus
 \begin{align}
|\nabla_{(\widetilde g(0))}X_{(\infty)}|_{ \widetilde g(0)} =\lim_{k\rightarrow\infty}|\nabla_{(g_{r_i})}X_{(i)}|_{g_{r_i}}=0,\notag
\end{align}
where the convergence is uniform on $B(p_{r_i},d_{0} ;  {g_{r_i}})$.
This means that $X_{(\infty)}$ is parallel. It remains to show that  $X_{(\infty)}$ is tangent to $\mathbb{R}$ in $\mathbb{R}\times\mathbb{S}^{n-1}$

Let $\frac{\partial}{\partial w}$ be the  vector field   tangent  to $\mathbb{R}$ in $\mathbb{R}\times\mathbb{S}^{n-1}$. By the construction of $\Phi_{r_i}$, we have
$$(\Phi_{r_i})_{\ast}(X_{(i)})=\frac{\partial}{\partial w}.$$
Then
 $$X_{(\infty)}=\lim_{i\to\infty}(\Phi_{r_i})_{\ast}(X_{(i)})=\frac{\partial}{\partial w}.$$
 Moreover, by  (\ref{identity-r}) and  Corollary \ref{uniform curvature decay},
 \begin{align}\label{x-maximal}
 |X_{(i)}|_{g_{r_i}}( x)=|\nabla f|(p_{r_i})
=\sqrt{R_{\rm max}}+o(1)>0, ~\forall~ x\in B(p_{r_i},d_{0} ;  {g_{r_i}}),
 \end{align}
as long as $r_i$ is large enough.  Thus, $X_{(\infty)}$ is nonzero and is tangent to $\mathbb{R}$ in $\mathbb{R}\times\mathbb{S}^{n-1}$. Since we have already proved that  $X_{(\infty)}$ is parallel,   $(\mathbb{R}\times\mathbb{S}^{n-1}, \widetilde g(0))$  must split off a line along  $\mathbb{R}$ in $\mathbb{R}\times\mathbb{S}^{n-1}$, and so does  the limit flow $\widetilde g(t)$. This proves  that
$N$ is diffeomorphic to   $\mathbb{S}^{n-1}$.

 Secondly,  we  prove  that $(\mathbb{S}^{n-1},g_{\mathbb{S}^{n-1}}(t))$ has    $\epsilon$-pinched Ricci curvature. In the other words,   Ricci curvature of $(M_{\infty},
\widetilde g(t))$ is  $\epsilon$-pinched along the vectors vertical to $X_{(\infty)}$.   By Gauss formula, we have
\begin{align*}
R(X,Y,Z,W)=\overline{R}(X,Y,Z,W)+\langle B(X,Z),B(Y,W)\rangle-\langle B(X,W),B(Y,Z)\rangle,
\end{align*}
where $X,Y,Z,W\in T\Sigma_{r}$ and $B(X,Y)=(\nabla_{X}Y)^{\bot}$. Note that
\begin{align*}
B(X,Y)=&\langle \nabla_{X}Y,\nabla f\rangle\cdot\frac{\nabla f}{|\nabla f|^{2}}\\
=&[\nabla_{X}\langle Y,\nabla f\rangle-\langle Y,\nabla_{X}\nabla f\rangle]\cdot\frac{\nabla f}{|\nabla f|^{2}}\\
=&-{\rm Ric}(X,Y)\cdot\frac{\nabla f}{|\nabla f|^{2}},
\end{align*}
Then
\begin{align}\label{ricci-convergence}
&R_{i'j'}=\overline{R}_{i'j'}\notag\\
&+R(\frac{\nabla f}{|\nabla f|},e_{i'},e_{j'},\frac{\nabla f}{|\nabla f|})-\frac{1}{|\nabla f|^{2}}\sum_k(R_{i'j'}R_{k'k'}-R_{i'k'}R_{k'j'}),
\end{align}
where  indices  $i',j',k'$ are corresponding  to a basis of   vector fields  on $T\Sigma_{r}$. Since $g_\infty(0)$  splits  along $X_{(\infty)}$, by the convergence of the rescaled metrics, we have
\begin{align}
R(\frac{\nabla f}{|\nabla f|},e_{i'},e_{j'},\frac{\nabla f}{|\nabla f|})=o(1)R_{i'j'},~\forall~x\in~B(p_{i},d_{0} ;  {g_{r_i}}).\notag
\end{align}
 On the other hand,  by (\ref{x-maximal}), it is easy to see
\begin{align}
0\le \frac{1}{|\nabla f|^{2}}\sum_k (R_{i'j'}R_{k'k'}-R_{i'k'}R_{k'j'})\le \frac{3}{2} R_{i'j'}.\notag
\end{align}
Thus by (\ref{ricci-convergence}) and  (\ref{pinching along level set}),  we get
\begin{align}
R_{i'j'}\ge \frac{1}{3}\overline{R}_{i'j'}\ge \frac{\epsilon}{3}Rg_{i'j'},~{\rm on}~B(p_{i},d_{0} ;  {g_{r_i}}),\notag
\end{align}
when $r_i$ are  large enough.
Since the above relation is independent of sequence of $\{p_i\}$ by Theorem \ref{theorem-soliton-real}, we prove in fact,
\begin{align}\label{original-pinching}
R_{i'j'}\ge \frac{1}{3}\overline{R}_{i'j'}\ge \frac{\epsilon}{3}Rg_{i'j'},~{\rm on}~B(p_{r},d_{0} ;  {g_{r}}),~\forall ~r>r_1>>1.
\end{align}

We show that (\ref{original-pinching}) holds along the flow $g_\tau(t)$.  By (\ref{mr-set}),  it is easy to  see
$$\{x\in M|~f(x)>r_1\}\subseteq \bigcup_{r>r_1}B(p_{r},d_{0} ;  {g_{r}}),$$
where $d_0\ge2\pi\sqrt{\frac{n-2}{\epsilon}}$.
Then by  (\ref{original-pinching}),
\begin{align}
R_{i'j'}\ge \frac{\epsilon}{3}Rg_{i'j'},~{\rm on}~\{x\in M|~f(x)>r_1\}.\notag
\end{align}
Since $\{x\in M|~f(x)\le r_1\}$ is compact and Ricci curvature is positive, we obtain
\begin{align}\label{pinching on the whole manifold}
R_{i'j'}(x)\ge \epsilon_0 Rg_{i'j'}(x),~\forall~x\in M,
\end{align}
for some $\epsilon_0>0$. Note that (\ref {pinching on the whole manifold}) is preserved under the metric scaling. Thus
(\ref{pinching on the whole manifold}) holds along the flow $g_\tau(t)$ for any $\tau>0$. By the convergence of  rescaled metrics, the Ricci curvature of $(M_{\infty},
\widetilde g(t))$ is  $\epsilon_0$-pinched along the vector fields vertical to $X_{(\infty)}$.  Therefore,  Ricci curvature of $(N=\mathbb{S}^{n-1},g_{N}(t))$ is   $\epsilon_0$-pinched.

By  the above  two steps,   we see that $ (N, g_{N}(t))$ is a $\kappa$-solution with  $\epsilon$-pinched Ricci curvature.
By Theorem \ref{classification of flow},    $ (N, g_{N}(t))$  must be a shrinking sphere flow. Note that $R^{(\infty)}(p_{\infty},0)=1$.  Then it is easy to see that  $(N, g_{N}(t))=g_{\mathbb{S}^{n-1}}(t)=(n-2)[(n-1)-2t]g_{\mathbb{S}^{n-1}}$.
The lemma is proved.

\end{proof}

In Lemma \ref{lem-for the theorem of pinching case-2},
 we prove the asymptotically cylindrical behavior of  the  steady soliton for  a special rescaling sequence with  base points $p_{r_i}$.   In the following, we  show that  the same result  holds  for an arbitrary sequence.

\begin{lem}\label{lem-for the theorem of pinching case-3}
Under the condition and notations of Lemma \ref{lem-for the theorem of pinching case-2}, let $p_i\to\infty$ be any sequence.   Then  by taking a subsequence of $p_i$ if necessary,   we have
\begin{align*}
(M,g_{p_{i}}(t),p_{i})\rightarrow(\mathbb{R}\times\mathbb{S}^{n-1},\widetilde{g}(t),p_{\infty}),~for~t\in(-\infty,0],
\end{align*}
where $g_{p_{i}}(t)=R(p_{i})g(R^{-1}(p_{i})t)$,  $(\mathbb{R}\times\mathbb{S}^{n-1},\widetilde{g}(t))$ is a shrinking cylinders flow  defined as in Lemma \ref{lem-for the theorem of pinching case-2}.
\end{lem}

\begin{proof}By Corollary \ref{uniform curvature decay},   there exists a  $C(\epsilon)$ such that
$$\frac{R(x)}{R(p)}\leq C(\epsilon),~\forall ~x, p\in \Sigma_{r}.$$
By  (\ref{pinching along level set}),  it follows
\begin{align*}
\overline{R}_{i'j'}(x)\geq \epsilon R(x)\overline{g}_{i'j'}\geq \epsilon  C(\epsilon)^{-1} R(p)\overline{g},
~\forall ~x, p\in \Sigma_{r}.
\end{align*}
Thus,  by the Myer's theorem,     we get as in
(\ref{diameter-mr}),
\begin{align}
{\rm diam}(\Sigma_{r},g)\le 2\pi\sqrt{\frac{(n-2)C(\epsilon)}{\epsilon R(p)}},~\forall~p\in \Sigma_{r}.\notag
\end{align}
Following the argument in Lemma \ref{lem-for the theorem of pinching case-1} (also see Remark \ref{remark} ),  we  see that for any $k\in\mathbb{N}$, there exists $r^{\prime}(k,\epsilon)$ such that for any $r\geq r^{\prime}(k,\epsilon)$ and $p\in \Sigma_{r}$ it holds
\begin{align}\label{diameter-estimate-21}
B(p,\frac{k}{\sqrt{R_{max}}};   {g_{p}})\subset M_{p,k}^{\prime}\subset B(p,2\pi\sqrt{\frac{(n-2)C(\epsilon)}{\epsilon}}+\frac{2k}{\sqrt{R_{max}}};  {g_{p}}),
\end{align}
where $g_{p}=R(p)g$ and $M_{p,k}^{\prime}$ is defined as
$$M_{p,k}^{\prime}=\{x\in M|~ f(p) -\frac{k}{\sqrt{R(p)}}\le f(x)\le f(p)+\frac{k}{\sqrt{R(p)}}\}.$$
Once  (\ref{diameter-estimate-21}) is true,  we can use  the argument  in the proof of  Lemma \ref{lem-for the theorem of pinching case-2}  to   the sequence   $(M,g_{p_{i}}(t),p_{i})$ to prove Lemma \ref{lem-for the theorem of pinching case-3}.
\end{proof}

As a corollary of Lemma \ref{lem-for the theorem of pinching case-3},  we get

\begin{cor}\label{cor of lem of theorem 1}
Under the condition of   Theorem \ref{theorem-pinching case}, we have
\begin{align}\label{eq:theorem-2}
\lim_{r\rightarrow\infty}\sup_{x,y\in \Sigma_{r}}|\frac{R(x)}{R(y)}-1|=0.
\end{align}
\end{cor}

\begin{proof}
Suppose that  (\ref{eq:theorem-2}) is not true. Then, we can find $\delta>0$ and two sequences $\{x_{i}\},  \{y_{i}\} \rightarrow\infty$ such that $f(x_{i})=f(y_{i})$ and
\begin{equation}\label{cor of lem4.4}
|\frac{R(x_{i})}{R(y_{i})}-1|>\delta,~as~i\rightarrow\infty.
\end{equation}
Applying  Lemma \ref{lem-for the theorem of pinching case-3} to the sequence  $\{x_{i}\}$, we get
$$(M,g_{x_{i}}(t),x_{i})\rightarrow(\mathbb{R}\times\mathbb{S}^{n-1},\widetilde{g}(t),x_{\infty}),  $$
where $g_{x_{i}}(t)=R(x_{i})g(R^{-1}(x_{i})t)$ and $\widetilde{g}(t)$ is defined as in Lemma \ref{lem-for the theorem of pinching case-3}.  By (\ref{diameter-estimate-21}), we have
\begin{align}
y_{i}\in \Sigma_{f(x_i)}\subset B(x_{i}, 2\pi\sqrt{\frac{(n-2)C(\epsilon)}{\epsilon}};g_{x_{i}}),
\end{align}
where $g_{x_{i}}=g_{x_{i}}(0)$.
According to the convergence of $(M,g_{x_{i}}(t),x_{i})$, we can find a sequence of diffeomorphism $\Phi_{i}$ such that  $\Phi_{i}^{\ast}(g_{x_{i}}(t))$ converges  to $\widetilde{g}(t)$ in $C_{\rm loc}^{\infty}$ topology and $\Phi_{i}(y_{i})\in B(x_{\infty},3\pi\sqrt{\frac{(n-2)C(\epsilon)}{\epsilon}};\widetilde{g}(0))$. As a  consequence, there is a subsequence $y_{i_{k}}$ such that $\Phi_{i_{k}}(y_{i_{k}})\rightarrow y_{\infty}$ in $B(x_{\infty},3\pi\sqrt{\frac{(n-2)C(\epsilon)}{\epsilon}}; $\newline $\widetilde{g}(0))$.   Note
$$\widetilde{R}(q,0)\equiv1, ~\forall~q\in \mathbb{R}\times\mathbb{S}^{n-1}.$$
Thus
$$\frac{R(y_{i_{k}})}{R(x_{i_{k}})}\rightarrow\widetilde{R}(y_{\infty},0)=1,~as~i_{k}\rightarrow\infty.$$
This  contradicts to (\ref{cor of lem4.4}). Hence, the corollary is true.
 \end{proof}

Since the sequence in Lemma \ref{lem-for the theorem of pinching case-3} and the sequence in Definition \ref{brendle} are  the same up to a scale,   to finish  the proof of Theorem \ref{theorem-pinching case}, we need to derive  the curvature decay property (i) in Definition \ref{brendle}.

\begin{proof}[Proof of Theorem \ref{theorem-pinching case}]
Let $\phi_{t}$ be  the  one parameter subgroup  generated by $-\nabla f$ and $g(t)=\phi_{t}^{\ast}g$ as before.
Let any  $p\in M$ such that $\nabla f(p)\neq0$ and $p_i=\phi_{t_{i}}(p)$, where $t_{i}\rightarrow-\infty$. Then by Lemma \ref{lem-for the theorem of pinching case-3},   we may assume that  sequence $(M, g_{r_{i}}(t), p_{i})$ converges to a shrinking cylinders flow $(\mathbb{R}\times\mathbb{S}^{n-1},\widetilde{g}(t))$, where
$$\widetilde{g}(t)=ds\otimes ds+(n-2)[(n-1)-2t]g_{\mathbb{S}^{n-1}}.$$
Thus   scalar curvature $\widetilde R(\cdot, t)$ of $\widetilde g( t)$ is given by
$$\widetilde R(\cdot, t)=\frac{n-1}{(n-1)-2t}.$$
It follows
$$\frac{\partial}{\partial t} \widetilde R(p_{\infty},0)=\frac{2}{n-1}.$$
On the other hand,  by the convergence of $(M,g_{r_{i}}(t), p_{i})$, we have
\begin{align}
\frac{\partial}{\partial t}\widetilde R(p_{\infty},0)=\lim_{i\rightarrow\infty} \frac{1}{R^{2}(p_{i},0)}\frac{\partial }{\partial t} R(p_{i},0)=\lim_{i_{k}\rightarrow\infty} \frac{1}{R^{2}(p,-t_{i_{k}})}\frac{\partial}{\partial t} R(p,-t_{i}).\notag
\end{align}
Hence we get
\begin{align}
\lim_{i\rightarrow\infty}F^{\prime}(t_{i})=\frac{n-1}{2},\notag
\end{align}
where $F(t)=R^{-1}(p,-t)$.
Since $t_i$ is an arbitrary sequence,
\begin{align}
\lim_{t \rightarrow-\infty}F^{\prime}(t)=\frac{n-1}{2},\notag
\end{align}
This implies
\begin{align}
\lim_{t\rightarrow\infty}\frac{1}{tR(p,-t)}=\lim_{t\rightarrow\infty}F^{\prime}(t)=\frac{n-1}{2}.\notag
\end{align}
Therefore, we derive
\begin{align}\label{scalar-curvature-t}
R(\phi_t(p)) |t| \rightarrow \frac{n-1}{2}, ~{\rm as}~t\rightarrow-\infty.
\end{align}

By the identity (\ref{identity-r}) and Corollary \ref{uniform curvature decay}, we have
\begin{align}\label{r-max} \lim_{t\rightarrow\infty}\frac{f(\phi_{-t}(p))}{t}=\lim_{t\rightarrow\infty}\frac{d f(\phi_{-t}(p))}{dt}=\lim_{t\rightarrow\infty}|\nabla f|^{2}=R_{\max}.
\end{align}
On the other hand, by  Corollary \ref{cor of lem of theorem 1},  it holds
\begin{align}
\lim_{r\rightarrow\infty}\sup_{x,y\in \Sigma_{r}}|\frac{R(x)}{R(y)}-1|=0.\notag
\end{align}
 Since $f(\phi_{t}(p))$ goes to $\infty$ as $t\to -\infty$,  we see that for any $x\in M\setminus\{o\}$ with $f(x)\ge r_0$,   there exists a  $t_{x}<0$  such that $f(x)=f(\phi_{t_{x}}(p))$. Thus
$$\lim_{f(x)\rightarrow\infty}|\frac{R(x)}{R(\phi_{t_{x}}(p))}-1|=0.$$
Combining  this with (\ref{scalar-curvature-t}) and (\ref{r-max}),  we deduce
$$\lim_{f(x)\rightarrow\infty}R(x)f(x)=\lim_{t_{x}\rightarrow-\infty}R(\phi_{t_{x}}(p))f(\phi_{t_{x}}(p))=\frac{n-1}{2}R_{\max}.$$
By (\ref{cao}), we finally  get  the property (i) in Definition \ref{brendle}.

\end{proof}

 If  (\ref{pinching along level set}) is  replaced  for the ambient metric $g$  by
 \begin{align}\label{m-pinching}
{\rm Ric}(v,v)\geq \epsilon R(x)g(v,v), ~\forall~ v\in T_{x}\Sigma_{r}, ~r>r_0,
\end{align}
we give another version of Theorem \ref{theorem-pinching case} as follows.

\begin{cor}\label{theorem-original pinching condition}
Any $\kappa$-noncollapsed steady Ricci soliton $(M,g,f)$  with nonnegative curvature operator, positive Ricci curvature and a uniform scalar curvature decay  must be rotationally symmetric, if $(M,g,f)$ admits a unique equilibrium point  and
 there exist  $r_0>0$ and $\epsilon>0$ such that the pinching  condition (\ref{m-pinching}) holds.
\end{cor}

\begin{proof}
By Theorem \ref{theorem-pinching case}, it suffices  to verify that  $(M,g,f)$ satisfies (\ref{pinching along level set}).
Let $e_{i'}(i'=1,2,\cdots,n-1)$ be  normal eigenvector fields of $\overline{\rm Ric} (\overline{g})$. Then we  have Gauss equation
\begin{align}\label{horizontal estimate-1}
&\overline{\rm R}_{i'i'}=R_{i'i'}\notag\\
&-R(\frac{\nabla f}{|\nabla f|},e_{i'},e_{i'},\frac{\nabla f}{|\nabla f|})-\frac{1}{|\nabla f|^2}\sum_{k'}(R_{i'i'}R_{k'k'}-R_{i'k'}R_{k'i'}).
\end{align}
Note that $R(x)$ decays uniformly.  We have
\begin{align}
|\nabla f|(x)\ge\frac{\sqrt{R_{\max}}}{2},~\forall~r(x)>r_0,\notag
\end{align}
for some $r_0>0$. Thus by (\ref{m-pinching}), it follows
\begin{align}
|\sum_{k'}(R_{i'i'}R_{k'k'}-R_{i'k'}R_{k'i'})|\le& R_{i'i'}R+\sum_{k'}R_{i'k'}R_{k'i'}\notag\\
\le&R_{i'i'}R+R^2
\le(1+\frac{1}{\epsilon})R_{i'i'}R
=o(1)R_{i'i'}.\label{horizontal estimate-2}
\end{align}
Also we have
\begin{align}\label{f-ei-curvature}
|R(\frac{\nabla f}{|\nabla f|},e_{i'},e_{i'},\frac{\nabla f}{|\nabla f|})|\le&|\sum_{i'=1}^{n-1} R(\frac{\nabla f}{|\nabla f|},e_{i'},e_{i'},\frac{\nabla f}{|\nabla f|})|\notag\\
=&|{\rm Ric}(\frac{\nabla f}{|\nabla f|},\frac{\nabla f}{|\nabla f|})|\notag\\
=&\frac{|\langle\nabla R,\nabla f\rangle|}{2|\nabla f|^2}\notag\\
=&\frac{|\Delta R+2|{\rm Ric}|^2|}{2|\nabla f|^2}.
\end{align}

On the other hand, by
 Proposition 3.8 in \cite{DZ2}, we see that there exists a constant $C>0$ such that
\begin{align}
\frac{|\Delta R|(x)}{R^{2}(x)}\le C,~\forall~x\in M.\notag
\end{align}
Then by (\ref{f-ei-curvature}) and  (\ref{m-pinching}), we get
\begin{align}
&|R(\frac{\nabla f}{|\nabla f|},e_{i'},e_{i'},\frac{\nabla f}{|\nabla f|})|\notag\\
&\le\frac{(C+2)R^{2}}{|\nabla f|^2}
\le\frac{(C+2)R}{\epsilon|\nabla f|^2}R_{i'i'}
=o(1)R_{i'i'}.\notag
\end{align}
Combining this with (\ref{horizontal estimate-1}) and (\ref{horizontal estimate-2}),   we obtain
\begin{align}
\overline{\rm R}_{i'i'}(x)\ge \frac{1}{2}{\rm R}_{i'i'}(x)\ge \frac{\epsilon}{2}R(x),~{\rm as}~r(x)>r_0.\notag
\end{align}
This implies (\ref{pinching along level set}).
\end{proof}

\section{Proof of Theorem \ref{main-theorem-1}}
In this section, we  apply Theorem   \ref{theorem-pinching case}  to prove Theorem \ref{main-theorem-1}  by verifying   (\ref{pinching along level set})    for steady  Ricci solitons  under  the curvature decay  condition (\ref{special curvature decay}).

We begin with

\begin{lem}\label{lem-for main theorem-1} Let  $p_{i}\in M\rightarrow\infty$ and $(M, g_i(t),p_{i})$ be a sequence of rescaling Ricci flows,  where $g_i(t)=R(p_{i})g(R^{-1}(p_{i})t)$. Suppose that $(M, g_i(t),p_{i})$
 converges to a limit flow  $(M_{\infty},g_{\infty}(t),p_{\infty})$ in the Cheeger-Gromov topology. Then  scalar curvature $R^{(\infty)}(x,t)$ of $g_{\infty}(t)$ satisfies
\begin{align}\label{asymtotic-r-t}(1-\frac{t}{C_0})R^{(\infty)}(x,t)\equiv1,~\forall~x\in M_{\infty},~t\in (-\infty,0].
\end{align}
if  $(M,g,f)$  admits  a unique equilibrium point and it  satisfies (\ref{special curvature decay}).

\end{lem}

\begin{proof} First we note that  $R(g)$  decays uniformly by (\ref{special curvature decay}) and  (\ref{cao}). For any $x\in
B(p_\infty, g_\infty(0);  {r_0})\subset M_\infty$, we  choose  a sequence of $x_{i}\in B(p_{i},2r_{0};  {g_{{i}}})$ which  converges  to $x$ in Cheeger-Gromov topology, where $g_i= R(p_{i})g$.  Since the proof of relation (\ref{relation-1})  is still true for any $k>0$  when $p_r$ is replaced by any $p_i$,   we have
$$B(p_{i},2r_{0};  {g_{{i}}})\subset\{y\in M|~ f(p_{i})-\frac{2r_{0}\sqrt{R_{\max}}}{\sqrt{R(p_i)}}\le f(y)\le f(p_{i})+\frac{2r_{0}\sqrt{R_{\max}}}{\sqrt{R(p_i)}}\}.$$
By  (\ref{special curvature decay}), it follows
\begin{align}
\label{y-p-i-order}\lim_{i\rightarrow\infty}\frac{f(x_i)}{f(p_{i})}=1.
\end{align}
Thus
\begin{align}\label{limit-r}R^{(\infty)}(x,t)=\lim_{i\rightarrow\infty}\frac{R(x_i,R^{-1}(p_{i})t)}{R(p_{i})}
=\lim_{i\rightarrow\infty} \frac{R(x_{i},R^{-1}(p_{i})t)}{R(x_{i})}.\end{align}

Fix a point  $p\in M\setminus \{o\}$. Then (\ref{r-max}) holds.
Choose  $\tau_{i}<0$  such that $f(\phi_{\tau_{i}}(p))=f(x_{i})$. We claim
\begin{align}\label{same-order}
\lim_{i\rightarrow\infty}\frac{f(\phi_{R^{-1}(p_{i})t}(x_i))}{f(\phi_{\tau_{i}+R^{-1}(p_{i})t}(p))}=1,~\forall t\le 0.
\end{align}
It suffices to consider the case $t<0$.  By (\ref{special curvature decay}) and  (\ref{y-p-i-order}), we have
\begin{align}
\frac{f(\phi_{\tau_{i}+R^{-1}(p_{i})t}(p))}{|R^{-1}(p_{i})t|}=&\frac{f(\phi_{\tau_{i}+R^{-1}(p_{i})t}(p))
-f(\phi_{\tau_i}(p))}{|R^{-1}(p_{i})t|}+\frac{f(x_i)}{|R^{-1}(p_{i})t|}\notag\\
= &\frac{|\int_{R^{-1}(p_{i})t}^0|\nabla f|^2ds|}{|R^{-1}(p_{i})t|}+\frac{f(x_i)}{|R^{-1}(p_{i})t|}\notag\\
\to&R_{\max}(1+\frac{C_0}{|t|}),~{\rm as} ~i\to \infty.\notag
\end{align}
Similarly,
\begin{align}
\frac{f(\phi_{R^{-1}(p_{i})t}(x_i))}{|R^{-1}(p_{i})t|}=&\frac{f(\phi_{R^{-1}(p_{i})t}(x_i))
-f(x_i)}{|R^{-1}(p_{i})t|}+\frac{f(x_i)}{|R^{-1}(p_{i})t|}\notag\\
\to&R_{\max}(1+\frac{C_0}{|t|}),~{\rm as} ~i\to \infty.\notag
\end{align}
Thus (\ref{same-order}) comes from  the above two relations.

  By (\ref{special curvature decay}) and the identity (\ref{identity-r}), we have
\begin{align} \lim_{t\rightarrow\infty}\frac{f(\phi_{-t}(p))}{t}=\lim_{t\rightarrow\infty}\frac{d f(\phi_{-t}(p))}{dt}=\lim_{t\rightarrow\infty}|\nabla f|^{2}=R_{\max}.\notag
\end{align}
 Combining this with   (\ref{special curvature decay}), (\ref{same-order})  and  (\ref{y-p-i-order}),
we obtain from (\ref{limit-r}),
\begin{align*}
R^{(\infty)}(x,t)
=&\lim_{i\rightarrow\infty} \frac{f(x_i)}{f(\phi_{R^{-1}(p_{i})t}(x_i))}\\
=&\lim_{i\rightarrow\infty} \frac{f(\phi_{\tau_{i}}(p))}{f(\phi_{\tau_{i}+R^{-1}(p_{i})t}(p))}\\
=&\lim_{i\rightarrow\infty} \frac{\tau_{i}}{\tau_{i}+R^{-1}(p_{i})t}\\
=&\lim_{i\rightarrow\infty} \frac{\tau_{i}}{\tau_{i}+f(x_{i})t/(C_0\cdot R_{\max})}\\
=&\lim_{i\rightarrow\infty} \frac{\tau_{i}}{\tau_{i}+f(\phi_{\tau_{i}}(p))t/(C_0\cdot R_{\max})}\\
=&\frac{C_0}{C_0-t}.
\end{align*}
This proves (\ref{asymtotic-r-t}).

\end{proof}

\begin{proof}[Proof of Theorem \ref{main-theorem-1}] It  suffices   to verify that horizontally   Ricci curvature of $(M,g)$  is $\epsilon$-pinched for some $\epsilon>0$.   We use the contradiction argument.  Then there exist  a sequence of points $p_{i}\rightarrow\infty$ and vectors $v^{(i)}\in T_{p_{i}}\Sigma_{f(p_{i})}$ such that
\begin{align}\label{eq:4-1}
\frac{{\rm \overline{Ric}}(v^{(i)},v^{(i)})}{R(p_{i})\overline{g}(v^{(i)},v^{(i)})}\rightarrow0,~as~i\rightarrow\infty.
\end{align}
By Theorem \ref{theorem-soliton-real}, we may assume that $(M, g_i(t),p_{i})$ converges  to $(M_{\infty},  \widetilde g(t),$ $~p_{\infty})$, where $g_i(t)=R(p_{i})g(R^{-1}(p_{i})t)$, $(M_{\infty}, g_{\infty}(t))=(\mathbb{R} \times N,  dr^2+ g_N(t))$.

Let $X_{(i)}=R(p_{i})^{-\frac{1}{2}}\nabla f$.  Then, for any fixed $r>0$, we have
\begin{align}
\lim_{i\rightarrow\infty}\sup_{B(p_{i},r;  {g_{{i}}})}|X_{(i)}|_{g_{{i}}}=\lim_{i\rightarrow\infty}\sup_{B(p_{i},r; g_{{i}})}\sqrt{R_{\max}-R(x)}=\sqrt{R_{\max}},\notag \end{align}
where $g_{{i}}=R(p_{i})g$.
It follows
 \begin{align}\sup_{B(p_{i},r;   {g_{{i}}})}|\nabla^{k}_{(g_{{i}})}X_{(i)}|_{g_i}\le C_0\sup_{B(p_{i},r;  g_i)}|\nabla^{k-1}_{(g_i)}{\rm Ric}|_{g_i}\le C.\notag
\end{align}
 Thus we may assume that $X_{(i)}$ converges  to $X_{(\infty)}$. On the other hand, by Lemma \ref{lem-for main theorem-1}, we have
$$R^{(i)}(p,0)\rightarrow1, ~\forall~p\in B(p_{i},r;   g_{i}),$$
where $R^{(i)}$ are scalar curvatures of $g_i$.
Hence  similar to (\ref{x-parallel}), we have
$$|\nabla_{(\infty)}X_{(\infty)}|_{g_{\infty}(0)}=\lim_{i\rightarrow\infty}|\nabla_{(g_i)}X_{(i)}|_{g_{i}}=\lim_{i\rightarrow\infty}\frac{|{\rm Ric}|}{R^{\frac{1}{2}}(p_{i})}=0.$$
This implies that the limit manifold  will split off a real line along $X_{(\infty)}$.  As a consequence,  $X_{(\infty)}$ is tangent to $\mathbb{R}$ in $\mathbb{R}\times N$.

Now, we prove that $N$ is compact.  By Lemma \ref{lem-for main theorem-1}, we have
\begin{align}
\Delta_{\widetilde g(t)}R^{(\infty)}(q,t)\equiv0,\notag\\
\frac{\partial R^{(\infty)}(q,t)}{\partial t}=\frac{1}{C_0}(R^{(\infty)}(q,t))^2.\notag
\end{align}
 By flow equation
\begin{align}
\frac{\partial R^{(\infty)}(q,t)}{\partial t}=\Delta_{(\infty)}R^{(\infty)}(q,t)+2|{\rm Ric}^{(\infty)}|^2(q,t),\notag
\end{align}
It follows
\begin{align}\label{eq:eigenvalue}
(R^{(\infty)}(q,t))^2=2C_0|{\rm Ric}^{(\infty)}|^2(q,t).
\end{align}
Let $\lambda_{1}(q,t)\leq \lambda_{2}(q,t)\leq\cdots\leq \lambda_{n-1}(q,t)$ be the eigenvalues of ${\rm Ric}(g_{N}(t))$. By   the condition $C_0>\frac{n-2}{2}$,   we see that there exists a constant $\delta(n)>0$ such that
\begin{align}\label{pinching inequality}
\frac{\lambda_{1}(q,t)}{\lambda_{n-1}(q,t)}\geq \delta(n).
\end{align}
This implies that $N$ is compact.

At last,  we   check  the pinching condition. Choose an orthonormal basis $\{e_{1},\cdots,e_{n}\}$ w.r.t $g_{\infty}(0)$ of $T_{p_{\infty}}M_{\infty}$ such that $e_{n}=\frac{X_{(\infty)}}{ |X_{(\infty)}|_{ g_{\infty}(0)}}$. By the convergence of $X_{(i)}$, we can choose a sequence of orthonormal bases $\{e_{1}^{(i)},\cdots,e_{n}^{(i)}\}$ w.r.t $g_{{i}}$ of $T_{p_{i}}M$ such that $e_{n}^{(i)}=\frac{X_{(i)}}{ |X_{(i)}|_{g_{i}}   }$,
$${\rm span}\{e_{1}^{(i)},\cdots,e_{n-1}^{(i)}\}=T_{p_{i}}\Sigma_{f(p_{i})}$$
and each $e_{k}^{(i)}$ converge to $e_{k}$ for $1\le k\le n-1$.
Then by (\ref{pinching inequality}), we have
\begin{align*}
&{\rm Ric}^{(\infty)}(e_{k},e_{k})\ge \frac{\delta(n)}{n}R^{(\infty)}(p_{\infty}),~\forall~1\le k\le n-1, \\
&{\rm Ric}^{(\infty)}(e_{k},e_{l})=0,~for ~k\neq l.
\end{align*}
By  the convergence, it follows
\begin{align}
&{\rm Ric}^{(i)}(e^{(i)}_{k},e^{(i)}_{k})\ge (1-\varepsilon_{i}) \frac{\delta(n)}{n}R^{(i)}(p_{i}),~\forall~1\le k\le n-1,\notag\\
&{\rm Ric}^{(i)}(e^{(i)}_{k},e^{(i)}_{l})\rightarrow0,~as~i\rightarrow\infty,~~\forall~k\neq l, \notag
\end{align}
where $\varepsilon_{i}\rightarrow 0$ as $i\rightarrow\infty$.
Since $v^{(i)}\in T_{p_{i}}\Sigma_{f(p_{i})}$, we get
\begin{align}
\frac{{\rm Ric}^{(i)}(v^{(i)},v^{(i)})}{g(v^{(i)},v^{(i)})}\ge \frac{\delta(n)}{2n}R^{(i)}(p_{i}), ~as~i\rightarrow\infty.\notag
\end{align}
On the other hand, since the limit manifold splits along $X_{(\infty)}$, we also have
\begin{align}
R^{(i)}(X_{(i)}, e_k^{i}, e_l^{i},X_{(i)})=o(1)R_{kl}^{(i)},~as~i\to\infty.\notag
\end{align}
  By (\ref{horizontal estimate-1}) and  (\ref{special curvature decay}),  it is easy to see
$$\overline{ R}_{kl}^{(i)}\ge \frac{1}{2} R_{kl}^{(i)}.$$
Thus
\begin{align}
\frac{{\rm \overline{Ric}}^{(i)}(v^{(i)},v^{(i)})}{g(v^{(i)},v^{(i)})}\ge \frac{{\rm Ric}^{(i)}(v^{(i)},v^{(i)})}{2g(v^{(i)},v^{(i)})}\ge\frac{\delta(n)}{4n}R^{(i)}(p_{i}), ~as~i\rightarrow\infty.\notag
\end{align}
This is a  contradiction to (\ref{eq:4-1}). Hence we complete the proof.
\end{proof}

\end{document}